\documentclass[12pt]{amsart}
\usepackage{amssymb,amsmath,amsthm}
\numberwithin{equation}{section}
\usepackage{enumerate}

\newtheorem{thm}{Theorem}[section]

\newtheorem{cor}[thm]{Corollary}
\newtheorem{prop}[thm]{Proposition}
\theoremstyle{definition}

\begin{document}
\title{Composition operators on Hardy spaces of a half plane}
\author[S.~J.~Elliott and M.~T.~Jury]{Sam Elliott$^1$ and Michael T. Jury$^2$}

\address{Department of Pure Mathematics\\
  University of Leeds\\
  Leeds\\
  LS2~9JT\\
  UK}

\email{samuel@maths.leeds.ac.uk}

\address{Department of Mathematics\\
  University of Florida\\
  Box 118105\\
  Gainesville, FL 32611-8105\\
  USA}

\email{mjury@math.ufl.edu}


\thanks{${}^1$Research supported by EPSRC grant EP/P502578/1}
\thanks{${}^2$Research supported by NSF grant DMS 0701268} 
\thanks{The authors would also like to thank the organizers of the Fourth Conference on Modern Complex Analysis and Operator Theory and Applications, held in El Escorial, Spain, June 2009, where this work was initiated.}
\date{\today}

\begin{abstract}
We prove that a composition operator is bounded on the Hardy space $H^2$ of the right half-plane if and only if the inducing map fixes the point at infinity non-tangentially, and has a finite angular derivative $\lambda$ there. In this case the norm, essential norm, and spectral radius of the operator are all equal to $\sqrt{\lambda}$.
\end{abstract}
\maketitle

\section{Introduction}
Analytic composition operators have already been studied in great detail on Hardy spaces on the unit disc of the the complex plane. It is a consequence of the Littlewood subordination principle that all such operators are bounded on all the Hardy spaces, as well as a large class of other spaces of functions. Characterizations of compactness and weak compactness have also been produced, including those of Cima and Matheson \cite{CimaMatheson}, Sarason \cite{Sarason} and Shapiro \cite{Shapiro87}.

Comparatively little work has been done, however, on equivalent spaces of a half plane. Although corresponding Hardy spaces of the disc and half-plane are isomorphic, composition operators act very differently in the two cases, for example, not all analytic composition operators are bounded. In addition, Valentin Matache showed in \cite{Matache99} that there are in fact no compact composition operators in the half-plane case. Characterizations of isometric composition operators have now been given in both cases however: in the disc by Nordgren \cite{Nordgren}, and more recently in the half plane by Chalendar and Partington \cite{Chalendar03}.

In \cite{Elliott}, the first named author produced a simple description of the bounded composition operators with rational symbol, and gave a necessary condition for boundedness of more general inducing maps. Previously, Matache \cite{Matache08} showed that a composition operator is bounded on the Hardy space if and only if it has finite angular derivative at infinity.  We will use techniques similar to those developed by the second named author in \cite{Jury} to give a simplified proof of Matache's result, and we prove that the norm of $C_\varphi$ on $H^2$ is in fact equal to the square root of the angular derivative $\varphi^\prime(\infty)$.  Furthermore the norm, essential norm, and spectral radius are all equal, which in particular strengthens Matache's result on non-compactness.  Similar results follow for the other $H^p$ spaces.  

\section{Preliminaries}
Let $\mathbb H$ denote the right half-plane $\{\Re z >0\}$.  We consider the Hardy spaces $H^p(\mathbb H)$ for $0<p<\infty$, defined as the space of all holomorphic functions in $\mathbb H$ for which
\begin{equation*}
\sup_{x>0} \frac{1}{2\pi}\int_{-\infty}^\infty |f(x+iy)|^p\, dy <\infty
\end{equation*}
When $p\geq 1$, the $p^{th}$ root of this supremum defines a norm, and $H^p(\mathbb H)$ is a Banach space.  The space $H^2(\mathbb H)$ is a reproducing kernel Hilbert space over $\mathbb H$, with kernel
\begin{equation*}
k_w(z) =\frac{1}{z+\overline{w}}
\end{equation*}
This means that for each $w\in\mathbb H$, the evaluation functional $f\to f(w)$ is bounded on $H^2$, and 
\begin{equation}\label{E:repro}
f(w)=\langle f, k_w\rangle_{H^2}.
\end{equation}
We will be interested in composition operators on $H^p$; these are operators defined by 
\begin{equation*}
C_\varphi f:=f\circ \varphi
\end{equation*}
where $\varphi:\mathbb H\to \mathbb H$ is a holomorphic mapping.  If $C_\varphi$ is bounded, then a quick calculation using (\ref{E:repro}) shows that 
\begin{equation}\label{E:cphikernel}
C^*_\varphi k_w =k_{\varphi(w)}.  
\end{equation}

We will prove that $C_\varphi$ is bounded on $H^2$ if and only if the map $\varphi$ fixes the point at infinity and has a finite angular derivative there.  The key idea in the proof is to express both the angular derivative hypothesis and the boundedness conclusion in terms of the positivity of certain kernels.  By a positive kernel on $\mathbb H$ we mean any function $K(z,w)$ on $\mathbb H\times \mathbb H$ with the property that
\begin{equation}
\sum_{i,j=1}^n c_i \overline{c_j} K(x_i, x_j) \geq 0
\end{equation} 
for all $n\geq 1$, all choices of scalars $c_1, \dots c_n$ and sets of points $x_1, \dots x_n\in \mathbb H$.  

We begin by collecting some relevant background material.  The following characterization of self-maps of $\mathbb H$ is well known; it follows immediately from the half-plane formulation of the Nevanlinna-Pick interpolation theorem.

\begin{prop}[Nevanlinna]\label{P:nev}  A holomorphic function $\psi$ in $\mathbb H$ has positive real part if and only if the kernel
\begin{equation*}
\frac{\psi(z)+\overline{\psi(w)}}{z+\overline{w}}
\end{equation*}
is positive.
\end{prop}
Next, we need the fact that if $K_1(z,w)$ and $K_2(z,w)$ are positive kernels, then the product $K_1(z,w)K_2(z,w)$ is also a positive kernel; this is just the Schur (or Hadamard) product theorem expressed in kernel language; see e.g. \cite[Appendix A]{agler-mccarthy}.

Finally, we recall some of the theory of angular derivatives; for a thorough discussion of this topic see \cite[Chapter 4]{Shapiro93}.  A sequence of points $z_n=x_n+iy_n$ in $\mathbb H$ is said to approach $\infty$ non-tangentially if $x_n\to \infty$ and the ratios $|y_n|/x_n$ are uniformly bounded.  We say a map $\varphi:\mathbb H\to \mathbb H$ fixes infinity non-tangentially if $\varphi(z_n)\to \infty$ whenever $z_n\to \infty$ non-tangentially, and write $\varphi(\infty)=\infty$.  If $\varphi(\infty)=\infty$, we say that $\varphi$ has a finite angular derivative at $\infty$ if the non-tangential limit
\begin{equation}\label{E:angdivH}
\lim_{z\to \infty} \frac{z}{\varphi(z)}
\end{equation}
exists and is finite; we write $\varphi^\prime(\infty)$ for this quantity.  If we let $\psi$ be the self-map of $\mathbb D$ conjugate to $\varphi$ via the Cayley transform $\tau(\zeta)=\frac{1+\zeta}{1-\zeta}$, that is $\psi = \tau^{-1} \circ\varphi\circ\tau$, then (\ref{E:angdivH}) is easily seen to be equivalent to the existence of the non-tangential limit
\begin{equation}\label{E:angdivD}
\lim_{\zeta\to 1} \frac{1-\psi(\zeta)}{1-\zeta}.
\end{equation}
Moreover the Julia-Caratheodory theorem \cite[p.57]{Shapiro93} says that this limit is also equal to the non-tangential limit of $\psi^\prime(\zeta)$ as $\zeta\to 1$, justifying the name.  We will need the following half-plane variant of the Julia-Caratheodory theorem:
\begin{prop}[Julia-Caratheodory theorem in $\mathbb H$] \label{P:jc}  
Let $\varphi:\mathbb H\to \mathbb H$ be holomorphic.  The following are equivalent:
\begin{enumerate}
\item $\varphi(\infty)=\infty$ and $\varphi^\prime(\infty)$ exists.  
\item $\displaystyle{\sup_{z\in\mathbb H} \frac{\Re z}{\Re \varphi(z)}<\infty}$
\item $\displaystyle{\limsup_{z\to \infty} \frac{\Re z}{\Re \varphi(z)}<\infty}$
\end{enumerate}
Moreover the quantities in (2) and (3) are both equal to the angular derivative $\varphi^\prime(\infty)$.
\end{prop}
\begin{proof} 
Transporting things to the disk, we find using the notation introduced above that
\begin{equation}\label{E:halfplane_julia}
\frac{\Re z}{\Re \varphi(z)}= \left(\frac{|1-\psi(\zeta)|^2}{1-|\psi(\zeta)|^2}\right)\left(\frac{1-|\zeta|^2}{|1-\zeta|^2}\right)
\end{equation}
so that (1) implies (2) is just the half-plane version of Julia's theorem, \cite[p.63]{Shapiro93}.  That (2) implies (3) is trivial.  To prove (3) implies (1), we again move the problem to the unit disk.    Let $z\to \infty$ in $\mathbb H$ along the real axis, so that $r=\tau^{-1}(z)$ tends radially to $1$ in $\mathbb D$ (in fact non-tangential convergence is sufficient here).  Letting $L$ denote the limit supremum in (3), we see from (\ref{E:halfplane_julia}) that
\begin{equation}
\limsup_{r\to 1} \frac{|1-\psi(r)|^2}{|1-r|^2} \frac{1-|r|^2}{1-|\psi(r)|^2}\leq L,
\end{equation}
so for some finite $M>L$ and all $r$ sufficiently close to 1,
\begin{equation}\label{E:julialemma}
\frac{|1-\psi(r)|^2}{|1-r|^2} \leq M  \frac{1-|\psi(r)|^2}{1-|r|^2}.
\end{equation}
We will be done if we can show that
\begin{equation}\label{E:jcinfdisk}
\liminf_{r\to 1} \frac{1-|\psi(r)|^2}{1-|r|^2} <\infty,
\end{equation}
since again by the disk version of Julia-Caratheodory \cite[p.57]{Shapiro93}, this implies that $\psi$ has a finite angular derivative at $1$, which is equivalent to (1).  

To prove (\ref{E:jcinfdisk}), observe that the numerator on the left hand side of (\ref{E:julialemma}) dominates $(1-|\psi(r)|)^2$, thus
$$
\frac{(1-|\psi(r)|)^2}{|1-r|^2} \leq M \left(\frac{1-|\psi(r)|^2}{1-|r|^2} \right)
$$
which implies (after some algebra, and using the fact that $r$ is real)
$$
\frac{1-|\psi(r)|^2}{1-|r|^2} \leq M\left( \frac{1+|\psi(r)|}{1+|r|}\right)^2.
$$
The right hand side remains bounded as $r\to 1$, which proves (\ref{E:jcinfdisk}).

\end{proof}

The key observation used in the proof of Theorem~\ref{T:main} below is that by combining Propositions \ref{P:nev} and \ref{P:jc}, we can express the existence of the angular derivative in terms of the positivity of a kernel.  In particular, we note that if $\varphi^\prime(\infty)=\lambda$, then $\displaystyle{\sup_{z\in\mathbb H} \frac{\Re z}{\Re \varphi(z)}=\lambda}$, and so the function $\varphi(z)-\lambda^{-1}z$ has positive real part in $\mathbb H$.  By Nevanlinna's theorem, the kernel
\begin{equation}
\frac{(\varphi(z)-\lambda^{-1}z)+\overline{(\varphi(w)-\lambda^{-1}w})}{z+\overline{w}}
\end{equation}
is therefore positive.

\section{Main results}
\begin{thm}\label{T:main}
Let $\varphi:\mathbb H\to \mathbb H$ be holomorphic.  The composition operator $C_\varphi$ is bounded on $H^2(\mathbb H)$ if and only if $\varphi$ has a finite angular derivative $0<\lambda<\infty$ at infinity, in which case $\|C_\varphi\|=\sqrt{\lambda}$.  
\end{thm}
\begin{proof}
Suppose $\varphi$ fixes $\infty$ and $\varphi^\prime(\infty)=\lambda$; it suffices to show
\begin{equation}\label{E:kernel_norm}
\lambda \langle k_w, k_z\rangle_{H^2}  - \langle C_\varphi^* k_w, C_\varphi^* k_z\rangle_{H^2} 
\end{equation}
is a positive kernel.  Indeed, if this is so then (recalling that the linear span of the kernel functions $k_w$ is dense) the densely defined operator $C_\varphi^*:k_w\to k_{\varphi(w)}$ is bounded by $\sqrt{\lambda}$.  It thus has a unique bounded extension to $H^2$, and this extension must be the adjoint of $C_\varphi$.  Expanding out (\ref{E:kernel_norm}), we get
\begin{equation*}
\left(\frac{\lambda}{z+\overline{w}} - \frac{1}{\varphi(z)+\overline{\varphi(w)}} \right)
\end{equation*}
Dividing through by $\lambda$, we will show that 
\begin{equation}\label{E:kernel_temp}
\frac{1}{z+\overline{w}} - \frac{\lambda^{-1}}{\varphi(z)+\overline{\varphi(w)}} 
\end{equation}
is positive.  Factor out $(\varphi(z)+\overline{\varphi(w)})^{-1}$ to obtain
\begin{align}
\frac{1}{z+\overline{w}} - \frac{\lambda^{-1}}{\varphi(z)+\overline{\varphi(w)}} &= \frac{1}{\varphi(z)+\overline{\varphi(w)}}\left(\frac{\varphi(z)+\overline{\varphi(w)}}{z+\overline{w}}-\lambda^{-1}\right) \\
&= \frac{1}{\varphi(z)+\overline{\varphi(w)}}\left(\frac{(\varphi(z)-\lambda^{-1}z) +\overline{(\varphi(w)-\lambda^{-1}w)}}{z+\overline{w}}  \right) 
\end{align}
The expression on the last line is a product of two kernels.  The kernel $(\varphi(z)+\overline{\varphi(w)})^{-1}$ is positive since it factors as $\langle k_{\varphi(w)},k_{\varphi(z)}\rangle_{H^2}$.  The kernel in parentheses is positive by the remark following Proposition \ref{P:jc}.  Thus (\ref{E:kernel_temp}) is positive, as desired.

For the converse, suppose $\|C_\varphi\|\leq M$.  Then for every $z\in\mathbb H$, 
\begin{equation*}
\frac{1}{2\Re \varphi(z)} =\|C_\varphi^*k_z\|^2 \leq M^2 \|k_z\|^2 = M^2 \frac{1}{2\Re z}
\end{equation*}
which implies that $\Re z/\Re\varphi(z)$ is bounded by $M^2$ in $\mathbb H$, and hence $\varphi$ fixes $\infty$ and has finite angular derivative by Proposition~\ref{P:jc}. 

For the norm computation, observe that the first part of the proof shows that $\|C_\varphi\|\leq \sqrt{\lambda}$, while the second part of the proof (and an application of the Julia-Caratheodory theorem) shows that $\|C_\varphi\|\geq \sqrt{\lambda}$. 
\end{proof}

The following corollary follows immediately from the proof.  The corresponding statement is false in the disk; see \cite{BourdonRetsek}.
\begin{cor}
Bounded composition operators on $H^2(\mathbb H)$ have the \emph{kernel supremum property}, that is,
\begin{equation}\label{E:ksp}
\|C_\varphi\|=\sup_{z\in\mathbb H}\frac{\|C_\varphi^* k_z\|}{\|k_z\|}.
\end{equation}
\end{cor}

Matache \cite{Matache99} proved that there are no nonzero compact composition operators on $H^2(\mathbb H)$.  The following corollary strengthens this result, and shows that in a sense composition operators on the half-plane are as non-compact as possible.   Recall that if $T$ is a bounded operator on Hilbert space $H$, its essential norm $\|T\|_e$ is the distance from $T$ to the set of compact operators on $H$, and $\|T\|_e=0$ if and only if $T$ is compact.  

\begin{cor}  For every bounded composition operator on $H^2(\mathbb H)$, we have $\|C_\varphi\|=\|C_\varphi\|_e$.  In particular, there are no nonzero compact composition operators on $H^2(\mathbb H)$.  
\end{cor}
\begin{proof}  The inequality $\|C_\varphi\|_e \leq \|C_\varphi\|$ is trivial.  For the reverse inequality, recall that 
\begin{equation*}
\|C_\varphi\|_e \geq \limsup_{z\to \infty} \frac{\|C_\varphi^* k_z\|}{\|k_z\|}
\end{equation*}
Indeed, given $\epsilon>0$ choose a compact $Q$ so that $\|C_\varphi-Q\|\leq \|C_\varphi\|_e +\epsilon$.  Since the normalized kernel functions $k_z/\|k_z\|$ tend weakly to $0$ as $z\to \infty$, we have
\begin{align}
\|C_\varphi^*\|_e+\epsilon &\geq \|C^*_\varphi-Q\| \\
                         &\geq \limsup_{z\to \infty}\frac{\|(C_\varphi^* -Q)k_z\|}{\|k_z\|}   \\
                         &= \limsup_{z\to \infty} \frac{\|C_\varphi^* k_z\|}{\|k_z\|} \\
                         &= \left(\limsup_{z\to\infty} \frac{\Re z}{\Re \varphi(z)}\right)^{1/2}\\
                         &=\sqrt{\lambda }
\end{align}
where the last equality follows by the Julia-Caratheodory theorem.  
\end{proof}

\begin{thm}
If $C_\varphi$ is bounded on $H^2$, then its spectral radius is equal to its norm.
\end{thm}
\begin{proof}
Let $\lambda=\varphi^\prime(\infty)$.  Letting $\varphi_n$ denote the $n^{th}$ iterate of $\varphi$, by the Julia-Caratheodory theorem $\varphi_n^\prime(\infty)=\lambda^n$.  Thus $\|C_\varphi^n\| =\lambda^{n/2}$, so taking $n^{th}$ roots and letting $n$ go to infinity we get $r(C_\varphi)=\sqrt{\lambda}$ as desired.   
\end{proof}

Using the inner-outer factorization of $H^p$ functions, the norm calculation for $C_\varphi$ may be extended to the full range $0<p<\infty$.  We leave the details to the interested reader.
\begin{cor}  Let $0< p <\infty$.  The composition operator $C_\varphi$ is bounded on $H^p$ if and only if it is bounded on $H^2$, in which case ${\|C_\varphi\|}_{H^p}=\lambda^{1/p}$, where $\lambda=\varphi^\prime(\infty)$.    
\end{cor}

\bibliographystyle{siam}
\bibliography{biblio}

\begin{thebibliography}{10}

\bibitem{agler-mccarthy}
{\sc J.~Agler and J.~E. McCarthy}, {\em Pick Interpolation and Hilbert function
  spaces}, vol. 44 of Graduate Studies in Mathematics, American Mathematical
  Society, Providence, RI, 2002.

\bibitem{BourdonRetsek}
{\sc P.~S. Bourdon and D.~Q. Retsek}, {\em Reproducing kernels and norms of
  composition operators}, Acta. Sci. Math. (Szeged), 127 (2001), pp.~387--394.

\bibitem{Chalendar03}
{\sc I.~Chalendar and J.~R. Partington}, {\em On the structure of invariant
  subspaces for isometric composition operators on {$H^2(\mathbb{D})$ and
  $H^2(\mathbb{C}_+)$}}, Archiv der Math., 81 (2003), pp.~193--207.

\bibitem{CimaMatheson}
{\sc J.~A. Cima and A.~L. Matheson}, {\em Essential norms of composition
  operators and \text{Aleksandrov} measures}, Pacific J. Math, 179 (1997),
  pp.~59--63.

\bibitem{Elliott}
{\sc S.~J. Elliott}, {\em Adjoints of composition operators on hardy spaces of
  the half-plane}, J. Funct. Anal., 256 (2009), pp.~4162--4186.

\bibitem{Jury}
{\sc M.~T. Jury}, {\em Reproducing kernels, de {B}ranges-{R}ovnyak spaces, and
  norms of weighted composition operators}, Proc. Amer. Math. Soc., 135 (2007),
  pp.~3669--3675 (electronic).

\bibitem{Matache99}
{\sc V.~Matache}, {\em Composition operators on \text{Hardy} spaces of a
  half-plane}, Proc. Amer. Math. Soc., 127 (1999), pp.~1483--1491.

\bibitem{Matache08}
\leavevmode\vrule height 2pt depth -1.6pt width 23pt, {\em Weighted composition
  operators on {$H\sp 2$} and applications}, Complex Anal. Oper. Theory, 2
  (2008), pp.~169--197.

\bibitem{Nordgren}
{\sc E.~A. Nordgren}, {\em Composition operators}, Canad. J. Math., 20 (1968),
  pp.~442--449.

\bibitem{Sarason}
{\sc D.~Sarason}, {\em \text{Composition operators as integral operators}},
  Dekker, New York, 1990.

\bibitem{Shapiro87}
{\sc J.~H. Shapiro}, {\em The essential norm of a composition operator}, Ann.
  Math., 125 (1987), pp.~375--404.

\bibitem{Shapiro93}
\leavevmode\vrule height 2pt depth -1.6pt width 23pt, {\em \text{Composition
  operators and classical function theory}}, University texts in Mathematics:
  Tracts in Mathematics, Springer-Verlag, New York, 1993.

\end{thebibliography}

\end{document}